\documentclass[11pt]{article}
\usepackage{amscd}
\usepackage{amsfonts}
\usepackage{amsmath}
\usepackage{amssymb}
\usepackage{amsthm}
\usepackage{bbm}
\usepackage{CJK}
\usepackage{fancyhdr}
\usepackage{graphicx}
\usepackage{hyperref}
\usepackage{indentfirst}
\usepackage{latexsym}
\usepackage{mathrsfs}
\usepackage{xypic}

\newtheorem{theorem}{Theorem}[section]
\newtheorem{lemma}[theorem]{Lemma}
\newtheorem{definition}[theorem]{Definition}
\newtheorem{proposition}[theorem]{Proposition}
\newtheorem{example}[theorem]{Example}

\usepackage[top=1in,bottom=1in,left=1.25in,right=1.25in]{geometry}
\def\<{\langle}
\def\>{\rangle}
\def\a{\alpha}

\def\c{\cdot}

\def\o{\otimes}

\date{}
\begin{document}
\renewcommand{\baselinestretch}{1.2}
\renewcommand{\arraystretch}{1.0}
\title{\bf  On Nijenhuis Lie triple systems}
\date{}
\author{{\bf Shuangjian Guo$^{1}$, Bibhash Mondal$^{2}$, Ripan Saha$^{3}$\footnote
        { Corresponding author:~~Email: ripanjumaths@gmail.com}}\\
{\small 1.  School of Mathematics and Statistics, Guizhou University of Finance and Economics} \\
{\small   Guiyang  550025, P. R. of China}\\
{\small 2. Department of Mathematics, Behala College}\\
{\small Behala 700060, Kolkata, India}\\
{\small 3. Department of Mathematics, Raiganj University}\\
{\small  Raiganj 733134, West Bengal, India}
}
 \maketitle
\begin{center}
\begin{minipage}{12.cm}
\begin{center}{\bf ABSTRACT}\end{center}

In this paper, we investigate the mathematical structure of Nijenhuis Lie triple systems, an extension of classical Lie triple systems augmented with the Nijenhuis operator. Our study focuses on the cohomology of Nijenhuis Lie triple systems and demonstrates how abelian extensions of Nijenhuis Lie triple systems are related to cohomology groups. Additionally, we define Nijenhuis Lie triple 2-systems and also classify `strict' and `skeletal' Nijenhuis Lie triple 2-systems in terms of crossed modules and the cohomology of Nijenhuis Lie triple systems.
\medskip

\medskip

\end{minipage}
\end{center}
\normalsize\vskip0.5cm
\section{Introduction}
Lie triple systems, a class of mathematical structures, play a fundamental role in various areas of mathematics and theoretical physics. The concept of Lie triple systems was initially introduced by Jacobson\cite{J49} and further formulated by Yamguti\cite{Y60}. It also emerged in Cartan's exploration of Riemannian Geometry\cite{CK00}, where it was extensively developed for Symmetric spaces and related areas. Notably, the tangent space of a symmetric space constitutes a Lie triple system. These systems have found significant applications in physics, particularly in elementary particle theory and quantum mechanics, as well as in the numerical analysis of differential equations. They have become a captivating subject in mathematics, with their structure first studied by Lister\cite{L52}. For further exploration of Lie triple system, refer to \cite{CS, H23,HP02, KT04, Z04, ZCM16}.

 Dorfman \cite{Dorfman1993} investigated the Nijenhuis operator through the deformation of Lie algebras. Additionally, Nijenhuis operators on Lie algebras play a crucial role in the study of the integrability of nonlinear evolution equations \cite{Dorfman1993}. The introduction of Dirac structures by Dorfman provided new interpretations for existing Nijenhuis setups. In 2004, Gallardo and Nunes \cite{CN04} introduced Dirac Nijenhuis structures, while Longguang and Baokang \cite{LB04} independently developed Dirac Nijenhuis manifolds. In 2011, Kosmann-Schwarzbach \cite{kos} explored Dirac Nijenhuis structures on Courant algebroids.  In \cite{Wang}, the authors defined the concept of a Nijenhuis operator on a pre-Lie algebra, generating a trivial deformation of the pre-Lie algebra. In \cite{MS24}, the authors delved into Nijenhuis operators on Leibniz algebras from a cohomological standpoint. For further exploration of Nijenhuis operators in various other algebraic structures, refer to \cite{DS, HCM, LSZB}.

In this paper, we investigate Nijenhuis Lie triple systems, an extension of classical Lie triple systems augmented with the Nijenhuis operator. The combination of Lie triple systems with the Nijenhuis operator allows for a more nuanced investigation of the geometric and algebraic intricacies involved. 

We define the representation and cohomology of Nijenhuis Lie triple systems. The study of Rota-Baxter operators \cite{B60, JS21} and their modified counterparts \cite{MS24} across diverse algebraic structures has garnered significant attention. In this paper, we establish the relationship between Nijenhuis operators, Rota-Baxter operators, and modified Rota-Baxter operators within the context of Lie triple systems. We also study abelian extensions of Nijenhuis Lie triple systems and show that the isomorphism classes of abelian extensions are classified by the cohomology groups. The interplay between abelian extensions and cohomology provides a comprehensive picture of Nijenhuis Lie triple systems, enriching our comprehension of their algebraic and geometric intricacies.

Furthermore, we introduce the notion of Nijenhuis Lie triple 2-systems and demonstrate that skeletal Nijenhuis Lie triple 2-systems are classified by 5-cocycles of Nijenhuis Lie triple systems. We also establish a one-to-one correspondence between strict Nijenhuis Lie triple 2-systems and crossed modules of Nijenhuis Lie triple systems.

This paper is organized as follows: In Chapter \ref{chap-2}, we revisit the definition of Lie triple systems and delve into their cohomology, laying the groundwork for our subsequent discussions. Chapter \ref{chap-3} explores Nijenhuis Lie triple systems and their representations, offering insights into their distinctive features and structural properties. In Chapter \ref{chap-4}, we introduce a specialized cohomology theory tailored for Nijenhuis Lie triple systems, providing a deeper understanding of their algebraic and geometric properties. Chapter \ref{chap-5} shifts our focus to the study of abelian extensions, employing cohomology groups to explore the intricate relationship between extensions and the cohomology of Nijenhuis Lie triple systems. Finally, in Chapter \ref{chap-6}, we introduce the concept of Nijenhuis Lie triple 2-systems and elucidate the classification of skeletal Nijenhuis Lie triple systems through cocycles of our cohomology.

Throughout this paper, all vector spaces and (multi)linear maps are over an algebraically closed field $\mathbb{K}$ of characteristic 0.
\bigskip
\section{ Preliminaires}\label{chap-2}
\def\theequation{\arabic{section}.\arabic{equation}}
\setcounter{equation} {0}

We start with the background of Lie triple systems and their cohomology that we refer the reader to~ \cite{CM, L52, Y60} for more details.
\medskip

\begin{definition}
 A vector space $T$ together with a trilinear map $(x,y,z) \mapsto [x,y,z]$ is called a Lie triple system  if
\begin{eqnarray*}
&& [x, x, y]=0,\\
&& [x, y, z]+[y, z, x]+[z, x, y]=0,\\
&&  [x_1,x_2,[x_3,x_4,x_5]] = [[x_1,x_2, x_3], x_4, x_5] + [x_3,[x_1,x_2,x_4],x_5] + [x_3,x_4,[x_1,x_2, x_5]],
\end{eqnarray*}
for all $x, y, z, x_i \in T,  1\leq i \leq5$.
\end{definition}

\begin{example}
Let $(T, [\c, \c])$ be a Lie algebra. We define $[\c,\c,\c]: T\times T\times T\rightarrow T$ by
\begin{eqnarray*}
[x, y, z] = [[x, y], z], \forall x, y, z \in T.
\end{eqnarray*}
Then $(T, [\c, \c, \c])$ becomes a Lie triple system.
\end{example}

\begin{definition} Let $(T,[\cdot, \cdot, \cdot])$ be a Lie triple system,   $V$ an $\mathbb{K}$-vector space. If $\theta : T \times T \rightarrow \mathfrak{gl}(V)$ is a bilinear map such that for all $x_1, x_2, x_3,x_4\in T$,
\begin{eqnarray}
&&\theta(x_3, x_4)\theta(x_1, x_2)-\theta(x_2, x_4)\theta(x_1,x_3) - \theta(x_1,[x_2, x_3, x_4]) + D(x_2, x_3)\theta(x_1,x_4) = 0,\\
&&\theta(x_3, x_4)D(x_1, x_2)- D(x_1, x_2)\theta(x_3, x_4) + \theta([x_1, x_2, x_3], x_4) + \theta(x_3,[x_1, x_2, x_4]) = 0,~~~~~
\end{eqnarray}
where $D(x_1, x_2) = \theta(x_2, x_1)-\theta(x_1, x_2)$, then $(V,\theta)$ is called the representation of $(T,[\cdot, \cdot, \cdot])$. $(V, 0)$ is called the trivial representation of $(T,[\cdot, \cdot, \cdot])$.
\end{definition}

It follows that any Lie triple system $T$ is a representation of itself with
\begin{align*}
    \theta (x, y)z =  [ z, x, y],~~~~ D(x, y)z=[x, y, z].
\end{align*}
This is called the adjoint representation.

\medskip

Let $\theta$ be a representation of $(T,[\cdot, \cdot, \cdot])$ on $V$.   Denote by
\begin{eqnarray*}
C^{2n+1}(T, V):=\mathrm{Hom}(\underbrace{T\times \cdots \times T}_{2n+1}, V), n\geq 0,
\end{eqnarray*}
which is the space of $(2n+1)$-cochains.  For any $f \in C^{2n+1}(T, V)$ satisfies
\begin{eqnarray*}
&&f(x_1, \ldots, x_{2n-2},  x, x, y)=0,\\
&& f(x_1,\ldots, x_{2n-2}, x, y, z)+f(x_1,\ldots, x_{2n-2},  y, z, x)+f(x_1,\ldots, x_{2n-2}, z, x, y)=0,
\end{eqnarray*}

The coboundary operator $\delta^{2n-1}: C^{2n-1}(T, V) \rightarrow C^{2n+1}(T, V)$, $n\geqslant 1$ is given by
\begin{align*}
&\delta^{2n-1} f(x_1,\ldots,x_{2n+1})\\
&=\theta(x_{2n}, x_{2n+1})f(x_1,x_2,\ldots,x_{2n-1})-\theta( x_{2n-1}, x_{2n+1})f(x_1,x_2,\ldots,x_{2n-2}, x_{2n})\\
&\ \ \ \ \ \ +\sum_{i=1}^n(-1)^{i+n}D(x_{2i-1}, x_{2i}) f(x_1,\ldots,\widehat{x_{2i-1}},\widehat{x_{2i}},\ldots, x_{2n+1})\\
&\ \ \ \ \ \ +\sum_{i=1}^n\sum_{j=2i+1}^{2n+1}(-1)^{n+i+1} f(x_1,\ldots,\widehat{x_{2i-1}},\widehat{x_{2i}},\ldots, [x_{2i-1}, x_{2i}, x_j],\ldots, x_{2n+1}),
\end{align*}
for any $x_1, x_2, \ldots, x_{2n+1} \in T, f\in C^{2n-1}(T, V)$,  where $\widehat{}$ denotes omission.  The  $(2n+1)$-th cohomology group is  denoted by $H^\ast(T, V).$

\begin{definition}
A linear map $N: T\rightarrow T$ is called a Nijenhuis operator on a Lie triple system $(T, [\c, \c, \c])$ if $N$ satisfies the following condition
\begin{eqnarray*}
&& [Nx, Ny, Nz]=N\big( [Nx, Ny, z]+[x, Ny, Nz]+[Nx, y, Nz]\\
&&\ \ \ \ \ \ \ \ \ \ \ \ \ \ \ \ \ \ \ \ \ \  \ \  -N([Nx, y, z]+[x, Ny, z]+[x, y, Nz]-N[x, y, z])\big).
\end{eqnarray*}
\end{definition}

\begin{definition}
A linear map $R: T\rightarrow T$ is called a Rota-Baxter operator of weight $\lambda$ on a Lie triple system $(T, [\c, \c, \c])$ if $R$ satisfies the following condition
\begin{eqnarray*}
&& [Rx, Ry, Rz]\\
&=&R( [Rx, Ry, z]+[x, Ry, Rz]+[Rx, y, Rz]+\lambda [Rx, y, z]+\lambda[x, Ry, z]+\lambda[x, y, Rz]+\lambda^2[x, y, z]).
\end{eqnarray*}
\end{definition}

\bigskip
\section{ Nijenhuis Lie triple systems and  representations}\label{chap-3}
\def\theequation{\arabic{section}.\arabic{equation}}
\setcounter{equation} {0}
In this section, we study  Nijenhuis Lie triple systems and introduce the notion of their representations.
\begin{definition}
A Lie triple system $(T, [\c, \c, \c])$ equipped with a Nijenhuis operator $N$ on $T$ is called
a Nijenhuis  Lie triple system and it is denoted by $(T_N , [\c, \c, \c])$ or simply by $T_N$.
\end{definition}

\begin{definition}
 Let $(T_N , [\c, \c, \c])$ and $(T'_{N'} , [\c, \c, \c]')$ be two Nijenhuis  Lie triple systems. Then a linear  map $\phi: T\rightarrow T$ is called a morphism of Lie triple systems if the map $\phi$ is a Lie triple system morphism satisfying the condition $N'\circ \phi=\phi \circ N$.
\end{definition}
\begin{proposition}(\cite{CM})
 Let $(T_N , [\c, \c, \c])$ be a Nijenhuis  Lie triple system. Then  $(T, [\c, \c, \c]_N)$ is a Lie triple system, where \begin{eqnarray*}
[x, y, z]_N&=&[Nx, Ny, z]+[x, Ny, Nz]+[Nx, y, Nz]\\
&&  -N([Nx, y, z]+[x, Ny, z]+[x, y, Nz]-N[x, y, z]),
\end{eqnarray*}
and $N$ is a homomorphism from $(T, [\c, \c, \c]_N)$ and $(T, [\c, \c, \c])$.
\end{proposition}

\begin{example}
Consider 2-dimensional Lie triple system  $\mathbb{R}^2$  with a basis $\{e_1, e_2\}$ and the nonzero multiplication is given by
\begin{eqnarray*}
[e_1, e_2, e_2]=e_1.
\end{eqnarray*}
Consider the linear map $N: \mathbb{R}^2\rightarrow \mathbb{R}^2$  defined by $N(x) = Ax$,  where $A=\left(
                                            \begin{array}{ccc}
                                              a & b \\
                                              c & d
                                            \end{array}
                                          \right)$.
Then $N$ is a   Nijenhuis operator on $(T, [\cdot, \cdot, \c])$  if and only if
\begin{eqnarray*}
&&[Ne_i,Ne_j, Ne_k]= N\big( [Ne_i, Ne_j, e_k]+[e_i, Ne_j, Ne_k]+[Ne_i, e_j, Ne_k]\\
&&\ \ \ \ \ \ \ \ \ \ \ \ \ \ \ \ \ \ \ \ \ \  \ \ \  -N([Ne_i, e_j, e_k]+[e_i, Ne_j, e_k]+[e_i, e_j, Ne_k]-N[e_i, e_j, e_k])\big)   ~~~~~\forall i, j, k=1, 2.
\end{eqnarray*}
A direct calculation can be obtained that $N$ is a   Nijenhuis operator on $(T, [\cdot, \cdot, \c])$  if and only if
\begin{eqnarray*}
c=0\ \ \ \mathrm{and} \ \ ad(a+d)=0.
\end{eqnarray*}
\end{example}
\begin{definition}
A linear map $R: T\rightarrow T$ is called a modified Rota-Baxter operator of weight $\lambda$ on a Lie triple system $(T, [\c, \c, \c])$ if $R$ satisfies the following condition
\begin{eqnarray*}
&& [Rx, Ry, Rz]\\
&=&R( [Rx, Ry, z]+[x, Ry, Rz]+[Rx, y, Rz]-\lambda[x, y, z])+\lambda [Rx, y, z]+\lambda[x, Ry, z]+\lambda[x, y, Rz].
\end{eqnarray*}
\end{definition}
Recall  that  a linear map $R: T\rightarrow T$ is said to be  a modified  Rota-Baxter operator of  weight $\lambda$ on the  Lie algebra $(T, [\cdot, \cdot])$ if $R$ satisfies
\begin{eqnarray*}
[R(x), R(y)]=R([R(x), y]+[x, R(y)])+\lambda[x, y], \text{~for~} x, y\in T.
\end{eqnarray*}
By Example 2.2, we have
\begin{example}
Let $R$ be a modified Rota-Baxter operator of weight $\lambda$ on the  Lie algebra $(T, [\cdot, \cdot])$. Then $R$ is a
modified Rota-Baxter operator of weight $\lambda$ on the induced  Lie triple system $(T, [\cdot, \cdot, \c])$.
\end{example}
\begin{lemma}
Let $(T, [\c, \c, \c])$ be a Lie triple system. Then $R$ is a Rota-Baxter operator of  weight $\lambda$ if and only if $ 2R+\lambda \mathrm{id}$ is a modified Rota-Baxter operator of weight $-\lambda^2$.
\end{lemma}
\begin{proof}
For any $x, y, z\in T$, we have
\begin{eqnarray*}
&& [(2R+\lambda \mathrm{id})(x),  (2R+\lambda \mathrm{id})(y), (2R+\lambda \mathrm{id})(z)]\\
&=&[2Rx+\lambda x,  2Ry+\lambda y, 2Rz+\lambda z]\\
&=& 8[Rx, Ry, Rz]+4\lambda[Rx, Ry, z]+4\lambda[R(x), y, Rz]+4\lambda[x, Ry, Rz]\\
&&+2\lambda^2[x, y, Rz]+2\lambda^2[x, Ry, z]+2\lambda^2[Rx, y, z]+\lambda^3[x, y, z]\\
&=& (2R+\lambda \mathrm{id})( [(2R+\lambda \mathrm{id})x, (2R+\lambda \mathrm{id})y, z]+[x, (2R+\lambda \mathrm{id})y, (2R+\lambda \mathrm{id})z]+[(2R+\lambda \mathrm{id})x, y, (2R+\lambda \mathrm{id})z]\\
&&+\lambda^2[x, y, z])-\lambda^2 [(2R+\lambda \mathrm{id})x, y, z]-\lambda^2[x, (2R+\lambda \mathrm{id})y, z]-\lambda^2[x, y, (2R+\lambda \mathrm{id})z].
\end{eqnarray*}
\end{proof}
\begin{proposition}
Let $(T , [\c, \c, \c])$ be a   Lie triple system and  $N: T \rightarrow T$ be a linear operator. Then\\
(a) If $N^2 = 0$,  then $N$ is a Nijenhuis operator if and only if $N$ is a Rota-Baxter operator.\\
(b) If $N^2$ = $N$,  then $N$ is a Nijenhuis operator if and only if $N$ is a Rota-Baxter operator
of weight -1.\\
(c) If $N^2 = \pm\mathrm{id}$,  then $N$ is a Nijenhuis operator if and only if $N$ is a modified Rota-Baxter operator of weight $\mp1$.
\end{proposition}
\begin{proof}
(a)  Let $N^2 = 0$. Suppose $N$ is a Nijenhuis operator. Then for any $x, y, z \in T$,  we have
\begin{eqnarray*}
&& [Nx, Ny, Nz]=N\big( [Nx, Ny, z]+[x, Ny, Nz]+[Nx, y, Nz]\\
&&\ \ \ \ \ \ \ \ \ \ \ \ \ \ \ \ \ \ \ \ \ \  \ \  -N([Nx, y, z]+[x, Ny, z]+[x, y, Nz]-N[x, y, z])\big)\\
&&\ \ \ \ \ \ \ \ \ \ \ \ \ \ \ \ \ \ \ \ \ \  \ =N( [Nx, Ny, z]+[x, Ny, Nz]+[Nx, y, Nz]).
\end{eqnarray*}
Hence, $N$ is a Rota-Baxter operator. Proof of the converse part is similar.\\
(b)  Let $N^2 = N$. Suppose $N$ is a Nijenhuis operator. Then for any $x, y, z \in T$,  we have
\begin{eqnarray*}
&& [Nx, Ny, Nz]\\
&=&N\big( [Nx, Ny, z]+[x, Ny, Nz]+[Nx, y, Nz]\\
&&  -N([Nx, y, z]+[x, Ny, z]+[x, y, Nz]-N[x, y, z])\big)\\
&=&N( [Nx, Ny, z]+[x, Ny, Nz]+[Nx, y, Nz]-[Nx, y, z]-[x, Ny, z]-[x, y, Nz]+[x, y, z]).
\end{eqnarray*}
Hence, $N$ is a Rota-Baxter operator of weight -1. Similarly, the converse can be
shown.\\
(c) Let $N^2 =\mathrm{id}$. Suppose $N$ is a Nijenhuis operator. Then, for any $x, y, z \in T$,  we have
\begin{eqnarray*}
&& [Nx, Ny, Nz]\\
&=&N\big( [Nx, Ny, z]+[x, Ny, Nz]+[Nx, y, Nz]\\
&&  -N([Nx, y, z]+[x, Ny, z]+[x, y, Nz]-N[x, y, z])\big)\\
&=&N( [Nx, Ny, z]+[x, Ny, Nz]+[Nx, y, Nz]+[x, y, z])-[Nx, y, z]-[x, Ny, z]-[x, y, Nz].
\end{eqnarray*}
Hence, $N$ is a modified Rota-Baxter operator of weight -1.  In a similar way the
other cases can be shown.
\end{proof}

\begin{definition}
 Let $(T_N, [\c, \c, \c])$ be a Nijenhuis Lie triple system. A representation of $T_N$ is a
triple $(V, \theta, N_V )$, where $(V, \theta)$ is a representation of the Lie triple system
and $N_V : V \rightarrow V$ is a linear map satisfying the following condition
\begin{eqnarray}
\theta(Nx, Ny)N_V &=& N_V\big(\theta(N(x), N(y))+\theta(N(x), y)N_V +\theta(x, N(y))N_V \nonumber\\
&&- N_V \theta(N(x), y)-N_V \theta(x, N(y))-N_V \theta(x, y) N_V+N_V^2\theta(x, y)\big),
\end{eqnarray}
for any $x, y\in T$.
\end{definition}
\begin{proposition}
Let $(T_N, [\c, \c, \c])$ be a Nijenhuis Lie triple system and $(V, \theta, N_V )$  be a representation over it. Define a  map $\theta_N:\wedge^2 T \rightarrow \mathfrak{gl}(V)$ by
\begin{eqnarray*}
\theta_N(x,y):=\theta(N(x), N(y))-N_V(\theta(N(x), y) +\theta(x,N( y))-N_V \circ \theta(x,y)),
\end{eqnarray*}
Also,
\begin{eqnarray*}
D_N(x,y):=\theta_N(y, x)-\theta_N(x,y)=D(N(x), N(y))-N_V(D(N(x), y) +D(x,N( y))-N_V \circ D(x,y)),
\end{eqnarray*}
for any $x, y\in T$.
\end{proposition}
\begin{proof}
One can show that  $\theta_N$  defines a representation of the Lie triple system $(T_N ,[\cdot, \cdot, \cdot])$ on $V$ directly by a tedious computation.   Moreover, for any $x, y\in T$ and $u\in V$,  we have
\begin{eqnarray*}
&& \theta_N(N(x),N(y))N_{V}(u)\\
&=& \theta(N^2(x), N^2(y))N_{V}(u)-N_V(\theta(N^2(x),N(y)) N_{V}(u) +\theta(N(x),N^2(y))N_{V}(u)-N_V\circ \theta(N(x),N(y))N_{V}(u))\\
&=& N_V\big(\theta(N^2(x), N^2(y)) u +\theta(N^2(x), N(y))N_V(u) +\theta(N(x), N^2(y))N_V(u) \\
&&-N_V \theta(N^2(x), N(y)) u-N_V \theta(N(x), N^2(y))u-N_V \theta(N(x), N(y)) N_V(u)+N_V^2\theta(N(x), N(y))u\big)\\
&&-N^2_V\big(\theta(N^2(x), N(y)) u +\theta(N^2(x), y)N_V(u) +\theta(N(x), N(y))N_V(u)-N_V \theta(N^2(x), y) u\\
&&-N_V \theta(N(x), N(y)) u-N_V \theta(N(x), y) N_V(u)+N_V^2\theta(N(x), y)u+\theta(N(x), N^2(y)) u  \\
&&+\theta(N(x), N(y))N_V(u)+\theta(x, N^2(y))N_V(u)-N_V\theta(N(x), N(y)) u-N_V \theta(x, N^2(y)) u\\
&& -N_V \theta(x, N(y)) N_V(u)+N_V^2\theta(x, N(y))+\theta(N(x), N(y)) u -N_V(\theta(N(x), y)N_V(u) +\theta(x, N(y))N_V(u) \\
&&-N_V \theta(N(x), y) u-N_V \theta(x, N(y)) u-N_V\theta(x, y) N_V(u)+N_V^2\theta(x, y)u\big)\\
&=& N_V\big(\theta_N(N(x), N(y)) u +\theta_N(N(x), y)N_V(u) +\theta_N(x, N(y))N_V(u)\\
&&-N_V \theta_N(N(x), y) u-N_V \theta_N(x, N(y)) u-N_V \theta_N(x, y) N_V(u)+N_V^2\theta_N(x, y)u)\big),
\end{eqnarray*}
which shows $(V, \theta_N, N_V)$ is a representation of the  Nijenhuis Lie triple system $(T_N , [\cdot, \cdot, \cdot])$.
\end{proof}
\bigskip

\section{ Cohomology of Nijenhuis Lie triple systems}\label{chap-4}
\def\theequation{\arabic{section}.\arabic{equation}}
\setcounter{equation} {0}

In this section,  we define a cohomology theory for Nijenhuis Lie triple systems. This cohomology theory involves studying the behavior of Nijenhuis operators with coefficients in a suitable representation.

Let $(T_N, [\c, \c, \c])$ be a Nijenhuis Lie triple system with representation $(V, \theta, N_V )$. Now by Proposition 3.3 and 3.10, we get a new Nijenhuis  Lie triple system $(T_N , [\c, \c, \c]_N)$ with representation
$(V, \theta_N, N_V)$ induced by the Nijenhuis operator $N$. We consider the Yamaguti cochain complex of this induced Lie triple system $(T, [\c, \c, \c]_N)$ with representation $(V, \theta_N, N_V)$ as follows:

For any $n \geq 0$, define cochain groups $C^{2n+1}_{\mathrm{NO}}(T, V ):$= Hom$(T^{\otimes 2n+1}, V)$, For any $f \in C^{2n+1}_{\mathrm{NO}}(T, V)$ satisfies
\begin{eqnarray*}
&&f(x_1, \ldots, x_{2n-2},  x, x, y)=0,\\
&& f(x_1,\ldots, x_{2n-2}, x, y, z)+f(x_1,\ldots, x_{2n-2},  y, z, x)+f(x_1,\ldots, x_{2n-2}, z, x, y)=0,
\end{eqnarray*}

The coboundary operator $\partial^{2n-1}: C^{2n-1}_{\mathrm{NO}}(T, V) \rightarrow C^{2n+1}_{\mathrm{NO}}(T, V)$, $n\geqslant 1$ is given by
\begin{align*}
&\partial^{2n-1} f(x_1,\ldots,x_{2n+1})\\
&=\theta_N(x_{2n}, x_{2n+1})f(x_1,x_2,\ldots,x_{2n-1})-\theta_N( x_{2n-1}, x_{2n+1})f(x_1,x_2,\ldots,x_{2n-2}, x_{2n})\\
&\ \ \ \ \ \ +\sum_{i=1}^n(-1)^{i+n}D_N(x_{2i-1}, x_{2i}) f(x_1,\ldots,\widehat{x_{2i-1}},\widehat{x_{2i}},\ldots, x_{2n+1})\\
&\ \ \ \ \ \ +\sum_{i=1}^n\sum_{j=2i+1}^{2n+1}(-1)^{n+i+1} f(x_1,\ldots,\widehat{x_{2i-1}},\widehat{x_{2i}},\ldots, [x_{2i-1}, x_{2i}, x_j]_{N},\ldots, x_{2n+1}),
\end{align*}
for any $x_1, x_2, \ldots, x_{2n+1} \in T, f\in C^{2n-1}_{\mathrm{NO}}(T, V)$,  where $\widehat{}$ denotes omission.

The above map satisfies the condition $\partial^{2n+1}\circ \partial^{2n-1}= 0$ as it a coboundary map for the induced
Lie triple system $(T, [\c, \c, \c]_N)$. Therefore, $(C^{2n-1}_{\mathrm{NO}}(T, V ), \partial^{2n-1})$ is a cochain complex. This cochain
complex is called the cochain complex of the Nijenhuis operator $N$ and the corresponding
cohomology groups are called the cohomology of Nijenhuis operator $N$ with coefficients in
the representation $V$ and is denoted by $H^\ast_{\mathrm{NO}}(T, V).$

\begin{definition}
Let $(T_N , [\c, \c, \c])$ be a Nijenhuis Lie triple system with a representation $(V, \theta, N_V)$,  for $n \geq 1$,
we define a map $\phi^{2n-1}: C^{2n-1}(T, V ) \rightarrow C^{2n-1}_{\mathrm{NO}}(T, V)$  by
\begin{align*}
&\phi^{2n-1}(f)(x_1,x_2,\ldots,x_{2n+1})\\
&=f(N(x_1),N(x_2),\ldots,N(x_{2n+1}))-\sum_{i=1}^{2n-1} (N_V \circ f)(N(x_1),N(x_2),\ldots,N(x_{i-1}),x_i,N(x_{i+1}),\ldots,N(x_{2n-1}))\\
&+\sum_{i,j=1}^{2n-1}(N_V^2 \circ  f)(N(x_1),N(x_2),\ldots,N(x_{i-1}),x_i,N(x_{i+1}),\ldots, N(x_{j-1}),x_j,N(x_{j+1}),\ldots,N(x_{2n-1}))
\\&-\ldots+(-1)^{2n-1}(N_V^{2n-1} \circ f)(x_1,x_2,\ldots,x_{2n-1})
\end{align*}
\end{definition}
\begin{lemma}
We have
\begin{eqnarray*}
\partial^{2n-1}(\phi^{2n-1}(f))(x_1, x_2, x_3,\cdots, x_{2n+1}) = \phi^{2n+1} ( \delta^{2n-1}(f))(x_1, x_2, x_3, \cdots, x_{2n+1}),
\end{eqnarray*}
for any $f\in C^{2n-1}$(T, V ) and $x_1, \cdots, x_{2n+1}\in T$.
\end{lemma}
\begin{proof}
Direct computation gives the result.
\end{proof}

\begin{definition}
Let $(T_N , [\c, \c, \c])$ be a Nijenhuis Lie triple system with a representation $(V, \theta, N_V)$. Define
\begin{eqnarray*}
C^1_{\mathrm{NLie}}(T, V ) := C^1(T, V ) \ \ \mathrm{and}\ \ C^{2n+1}_{\mathrm{NLie}}(T, V ) := C^{2n+1}(T, V ) \oplus C^{2n-1}_{\mathrm{NO}} (T, V ), \forall n \geq  1.
\end{eqnarray*}
For $n\geq 2$,  define a map $d^{2n-1}: C^{2n-1}_{\mathrm{NLie}}(T, V )\rightarrow C^{2n+1}_{\mathrm{NLie}}(T, V )$ by
\begin{eqnarray*}
d^{2n-1}(f, g) = (\delta^{2n-1} f, \partial^{2n-3} g + (-1)^n\phi^{2n-1} f),\ \ \ \ \forall(f, g) \in C^{2n-1}_{\mathrm{NLie}} (T, V ),
\end{eqnarray*}
and for $n = 1$, define $d^1: C^{1}_{\mathrm{NLie}}(T, V )\rightarrow C^{3}_{\mathrm{NLie}}(T, V )$ by
\begin{eqnarray*}
d^1(f) = (\delta^1 f, \phi^1 f),\ \ \ \ \forall f \in C^{1}_{\mathrm{NLie}} (T, V ).
\end{eqnarray*}
\end{definition}
\begin{theorem}
The map $d^{2n-1}: C^{2n-1}_{\mathrm{NLie}}(T, V )\rightarrow C^{2n+1}_{\mathrm{NLie}}(T, V )$  satisfies $d^{2n+1}\circ d^{2n-1} = 0$.
\end{theorem}

\begin{proof}
For $n\geq 1$ and $(f, g) \in C^{2n-1}_{\mathrm{NLie}} (T, V )$, we have
\begin{eqnarray*}
d^{2n+1}\circ d^{2n-1} (f, g)&=& d^{2n+1}\big( \delta^{2n-1} f, \partial^{2n-3} g + (-1)^n\phi^{2n-1} f\big)\\
&=& \big(\delta^{2n+1}\delta^{2n-1} f,  \partial^{2n-1}(\partial^{2n-3} g + (-1)^n\phi^{2n-1} f) +(-1)^n  \phi^{2n+1} \circ \delta^{2n-1} f  \big)\\
&=& (0, 0).
\end{eqnarray*}

\end{proof}

Denote the cohomology group of this cochain complex by $H^{\ast}_{\mathrm{NLie}} (T, V)$, which is called
the cohomology of the Nijenhuis Lie triple system $T_N$ with coefficients in the representation $(V, \theta, N_V)$.

\bigskip

\section{Abelian extension of Nijenhuis Lie triple systems}\label{chap-5}

 In this section, we study abelian extensions of Nijenhuis Lie triple systems in terms of cohomology groups.

Let $(T_N, [\c, \c, \c])$ be a Nijenhuis Lie triple system, for any linear
operator $N_V$ on $V$ with the trivial bracket $\mu$ defined by $\mu : V \times V \times V \rightarrow V$ by $\mu(x, y, z) = 0$ for
all $x, y,  z\in V$ makes $(V_{N_V}, \mu)$ a Nijenhuis Lie triple system.

\begin{definition} An abelian extension  of $(T_N, [\c, \c, \c])$ by $(V_{N_V}, \mu)$  is an exact sequence of  morphisms of Nijenhuis Lie triple system
$$
\xymatrix@C=0.5cm{
  0 \ar[r] & (V_{N_V}, \mu) \ar[rr]^{i} && (\hat{T}_{\hat{N}}, [\c, \c, \c]_\wedge) \ar[rr]^{p} && (T_N,  [\c, \c, \c]) \ar[r] & 0 }
$$
where $V_{N_V}$ is an abelian ideal of $\hat{T}_{\hat{N}}$, i.e., $[V_{N_V}, V_{N_V}, \hat{T}_{\hat{N}}]_\wedge = [V_{N_V}, \hat{T}_{\hat{N}}, V_{N_V}]_\wedge = [\hat{T}_{\hat{N}}, V_{N_V}, V_{N_V}]_\wedge= 0$.
\end{definition}
A section of an abelian extension $(\hat{T}_{\hat{N}}, [\c, \c, \c]_\wedge)$ of $(T_N,  [\c, \c, \c])$ by $(V_{N_V}, \mu)$ is
a linear map $s: T_N \rightarrow  \hat{T}_{\hat{N}}$ such that $ps = \mathrm{id}$.

 \begin{definition} Two abelian extensions $(\hat{T}_{\hat{N_1}}, [\c, \c, \c]_{\wedge_1})$ and $(\hat{T}_{\hat{N_2}}, [\c, \c, \c]_{\wedge_2})$ are said to be   isomorphic if there is an isomorphism of Nijenhuis Lie triple systems   $\eta: (\hat{T}_{\hat{N_1}}, [\c, \c, \c]_{\wedge_1})\rightarrow (\hat{T}_{\hat{N_2}}, [\c, \c, \c]_{\wedge_2})$ of LietsDer pairs that makes the following diagram commutative
\begin{eqnarray*}
\aligned
\xymatrix{
0  \ar[rr] && (V_{N_V}, \mu)  \ar[d]^{\mathrm{id}_{V_{N_V}}}\ar[rr]^{i} & & (\hat{T}_{\hat{N_1}}, [\c, \c, \c]_{\wedge_1}) \ar[d]^{\eta} \ar[rr]^{p} & & (T_N,  [\c, \c, \c]) \ar[d]^{\mathrm{id}_{T_N}} \ar[rr] & & 0 \\
0 \ar[rr]& & (V_{N_V}, \mu)  \ar[rr]^{i'} && (\hat{T}_{\hat{N_2}}, [\c, \c, \c]_{\wedge_2}) \ar[rr]^{q} & & (T_N,  [\c, \c, \c])  \ar[rr] & & 0}\\
\endaligned
\end{eqnarray*}
\end{definition}
Define a bilinear map $\theta: T_N\times T_N\rightarrow \mathfrak{gl}(V)$ by
\begin{eqnarray}
\widetilde{\theta}(x, y)u=[u, s(x), s(y)]_{\wedge},\ \ \ \forall x, y\in T_N, u\in V.
\end{eqnarray}
Also, we have
\begin{eqnarray*}
\widetilde{D}(x, y)u&=&\widetilde{\theta}(y, x)u-\widetilde{\theta}(x, y)u\\
&=&[u, s(y), s(x)]_{\wedge}-[u, s(x), s(y)]_{\wedge},\ \ \ \forall x, y\in T_N, u\in V.
\end{eqnarray*}
Since $V$ is an abelian ideal of of $\hat{T}_{\hat{N}}$ and $\hat{N}s(x)-sN(x)\in V$, for any $x_1, x_2, x_3, x_4\in T$ and $u\in V$,
we have $s([x_2, x_3, x_4]_\wedge)-[s(x_2), s(x_3), s(x_4)]_\wedge \in V$, which indicates that
\begin{eqnarray}
&&[u, s(x_1), s([x_2, x_3, x_4]_\wedge)]_{\wedge}=[u, s(x_1), [s(x_2), s(x_3), s(x_4)]_\wedge)]_{\wedge},\\
&&[N_V(u), sN(x_1), sN(x_2)]_\wedge=[N_V(u), \hat{N}s(x_1), \hat{N}s(x_2)]_\wedge.
\end{eqnarray}
\begin{theorem}
With the above notations, $(V, \widetilde{\theta}, N_V )$ is a representation of $(T_N, [ \c, \c, \c])$.
\end{theorem}
\begin{proof}
For any $x_1, x_2, x_3, x_4\in T$ and $u\in V$, we have
\begin{eqnarray*}
&& \widetilde{\theta}(x_3, x_4)\widetilde{\theta}(x_1, x_2)-\widetilde{\theta}(x_2, x_4)\widetilde{\theta}(x_1,x_3) - \widetilde{\theta}(x_1,[x_2, x_3, x_4]) + \widetilde{D}(x_2, x_3)\widetilde{\theta}(x_1,x_4)\\
&=& [[u, s(x_1), s(x_2)]_\wedge, s(x_3), s(x_4)]_\wedge-[[u, s(x_1), s(x_3)]_\wedge, s(x_2), s(x_4)]_\wedge\\
&& -[u, s(x_1), [ s(x_2), s(x_3), s(x_4)]_\wedge]_\wedge+[[u, s(x_1),s(x_4)]_\wedge, s(x_3), s(x_2)]_\wedge\\
&&-[[u, s(x_1),s(x_4)]_\wedge, s(x_2), s(x_3)]_\wedge\\
&=& 0,
\end{eqnarray*}
which yields that (2.1) holds, similarly, (2.2) holds. Further, we have
\begin{eqnarray*}
&& \widetilde{\theta}(Nx, Ny)N_V(u)\\
&=& [N_V(u), sN(x), sN(y)]_{\wedge}\\
&=& [N_V(u), \hat{N}s(x), \hat{N}s(y)]_\wedge\\
&=& N_v\big( [u, \hat{N}s(x), \hat{N}s(y)]_\wedge+[N_V(u), \hat{N}s(x), s(y)]_\wedge+[N_V(u), s(x), \hat{N}s(y)]_\wedge\\
&& -N_V([N_V(u), s(x), s(y)]_\wedge+[u, \hat{N}s(x), s(y)]_\wedge+[u, s(x), \hat{N}s(y)]_\wedge-N_V[u, s(x), s(y)]_\wedge)\big)\\
&=& N_V\big(\widetilde{\theta}(N(x), N(y)) u +\widetilde{\theta}(N(x), y)N_V(u) +\widetilde{\theta}(x, N(y))N_V(u)\\
&&-N_V \widetilde{\theta}(N(x), y) u-N_V \widetilde{\theta}(x, N(y)) u-N_V \widetilde{\theta}(x, y) N_V(u)+N_V^2\widetilde{\theta}(x, y)u\big),
\end{eqnarray*}
which indicates that (3.1) holds. Therefore, $(V, \widetilde{\theta}, N_V )$ is a representation of $(T_N, [ \c, \c, \c])$.
\end{proof}

For any section $s$, we define maps $\psi: T\o T\o T \rightarrow V$ and  $\chi: T\rightarrow V$ by
$$
\psi(x, y, z):=[s(x), s(y), s(z)]_{\wedge}-s([x, y, z]),~\chi(x)=\hat{N}s(x)-sN(x),~~\forall x, y, z\in T.
$$

For any $x, y, z\in T$ and $u, v, w\in V$, define a linear map $N_{\chi}: T\oplus V\rightarrow T\oplus V$ and a trilinear map $[\c, \c, \c]_{\psi}$ by
\begin{align*}
&[x+u, y+v, z+w]_{\psi}=[x, y, z]+\psi(x, y, z)+\theta(y, z)u-\theta(x, z)v+D(x, y)w, \\
& N_{\chi}(x+u)=N(x)+N_{V}(u)+\chi(x).
\end{align*}
\begin{proposition}  The pair $(T \oplus V, [\c, \c, \c]_\psi)$ is a Nijenhuis Lie triple system if and only if $(\psi, \chi)$ is a 3-cocycle in the cohomology of  the Nijenhuis Lie triple system $(T_N, [\c, \c, \c])$ with coefficients in the representation $(V, \theta, N_V )$.
\end{proposition}
\begin{proof}
For any $x, y, z\in T$ and $u, v, w\in V$,  then $N_{\chi}$ is a Nijenhuis operator if and only if
\begin{eqnarray*}
&& [N_{\chi}(x+u), N_{\chi}(y+v), N_{\chi}(z+w)]_{\psi}\\
&=&N_{\chi}\big( [N_{\chi}(x+u), N_{\chi}(y+v), z+w]_{\psi}+[x+u, N_{\chi}(y+v), N_{\chi}(z+w)]_{\psi}+[N_{\chi}(x+u), y+v, N_{\chi}(z+w)]_{\psi}\\
&& -N_{\chi}([N_{\chi}(x+u), y+v, z+w]_\psi+[x+u, N_{\chi}(y+v), z+w]_\psi+[x+u, y+v, N_{\chi}(z+w)]_\psi\\
&&-N_\chi[x+u, y+v, z+w]_\omega)\big),
\end{eqnarray*}
which induces that
\begin{eqnarray}
&& \psi(Nx, Ny, Nz)+\theta(Ny, Nz)\chi(x) -\theta(Nx, Nz)\chi(y) +D(Nx, Ny)\chi(z)\nonumber\\
&=&N_V(\psi(Nx, Ny, z)+\psi(Nx, y, Nz)+\psi(x, Ny, Nz)-N_V(\psi(Nx, y, z)+\psi(x, Ny, z)+\psi(x, y, Nz))\nonumber\\
&& +N^2_V\psi(x, y, z))+N_V( \theta(Ny, z)\chi(x)-\theta(Nx, z)\chi(y)+D(Nx, y)\chi(z)\nonumber\\
&&+\theta(y, Nz)\chi(x)-\theta(x, Nz)\chi(y)+D(x, Ny)\chi(z)-N_V(\theta(y, z)\chi(x)-\theta(x, z)\chi(y)+D(x, y)(\chi(z))))\nonumber\\
&& +\chi([Nx, Ny, z]+[Nx, y, Nz]+[x, Ny, Nz]-N_V([Nx, y, z]+[x, Ny, z]+[x, y, Nz]-N_V[x, y, z])).\ \ \ \ \ \ \ \ \ \
\end{eqnarray}

Conversely, if $(\psi, \chi)$ is a 3-cocycle, then $\psi$ is a
3-cocycle of $T$ with coefficients in the representation $(V, \theta)$ and $\partial N_{\chi} + \phi\psi = 0$. By routine
calculation, $\partial N_{\chi} + \phi\psi = 0$. is equivalent to that (5.4) holds.
\end{proof}

\begin{theorem} The isomorphism classes of abelian extensions of $(T_N, [\c, \c, \c])$ by $(V_{N_V}, \mu)$ are classified by the  cohomology group $H^{3}_{\mathrm{NLie}} (T, V)$.
\end{theorem}
\begin{proof} Let $(\hat{T}_{\hat{N}_1},  [\c, \c, \c]_{\wedge_1})$ and $(\hat{T}_{\hat{N}_2},  [\c, \c, \c]_{\wedge_2})$  be two isomorphic   abelian extensions and the isomorphism is given by $\eta: \hat{T}_{\hat{N}_1}\rightarrow \hat{T}_{\hat{N}_2}$. Let $s_1: T_N\rightarrow \hat{T}_{\hat{N}_1}$ be a section of $p$. Then
$$
p'\circ (\eta\circ s_1)=(p'\circ \eta)\circ s=p\circ s_1=\mathrm{id}.
$$
This shows that $s_2:=\eta\circ s_1$ is a section of $p'$. Since $\eta$ is a morphism of Nijenhuis Lie triple systems, we have $\eta|_V = \mathrm{id}_V$. Thus,
\begin{align*}
\psi_2(x, y, z)&=[s_2(x), s_2(y), s_2(z)]_{\wedge_2}-s_2([x, y, z])\\
&=\eta([s_1(x), s_1(y), s_1(z)]_{\wedge_1}-s_1([x, y, z]))=\psi_1(x, y, z),
\end{align*}
and
\begin{align*}
\chi_2(x)&=\hat{N}_2(s_2(x))-s_2(N(x))=\hat{N}_2(\eta\circ s_1(x))-\eta\circ s_1(N(x))\\
&=\eta (\hat{N}_1(s_1(x))- s_1(N(x))=\chi_1(x).
\end{align*}
Therefore, isomorphic central extensions give rise to same 3-cocycle, hence, correspond to same element in $H^{3}_{\mathrm{NLie}} (T, V)$.

Conversely, let $(\psi_1,  \chi_1 )$ and $(\psi_2,  \chi_2)$ be two cohomologous 3-cocycles. Therefore, there exists a map $\gamma: T \rightarrow  V$ such that
$$
(\psi_1,  \chi_1)-(\psi_2,  \chi_2)=d\gamma.
$$
The Nijenhuis Lie triple system structures on $T \oplus V$  corresponding to the above 3-cocycles are isomorphic via the map $\eta: T \oplus V\rightarrow T \oplus V$ given by $\eta(x, u) = (x, \gamma(x)+u)$. Hence the proof is completed.
\end{proof}

\section{Skeletal Nijenhuis Lie triple 2-systems and crossed modules}\label{chap-6}

In this section, we introduce the notion of Nijenhuis Lie triple 2-systems-and show that skeletal Nijenhuis Lie triple systems are classified by 5-cocycles of Nijenhuis Lie triple systems.

In the following, we first recall the definition of Lie triple 2-systems from \cite{S23}.
\begin{definition}
A Lie triple 2-system is a quintuple $(T_0, T_1,h,l_3,l_5)$, where $h:T_1\longrightarrow T_0$
is a linear map, $l_3: T_i\wedge T_j\wedge T_k\longrightarrow T_{i+j+k}~~(0\leq
i+j+k\leq 1)$ are trilinear maps and $l_5:\wedge^{2}T_0
\wedge \wedge^{3}T_0\longrightarrow T_{1}$ is a
multilinear map, and for any $x_i\in T_0~(i=1,\cdot\cdot\cdot,7),a\in T_1$, the followings are satisfied:
 \begin{align*}
(L1)~& l_3(x,x,y)=0,~~l_3(x,x,a)=0,~~l_3(a,x,y)+l_3(x,a,y)=0\\
(L2)~& hl_3(a,y,z)=l_3(h(a),y,z),\\
(L3)~&l_3(h(a),b,x)=l_3(a,h(b),x),~~l_3(h(a),x,b)=l_3(a,x,h(b)),~~l_3(x,h(a),b)=l_3(x,a,h(b)),\\
(L4)~&l_3(x,y,z)+l_3(y,z,x)+l_3(z,x,y)=0,~~l_3(x,y,a)+l_3(y,a,x)+l_3(a,x,y)=0,\\
(L5)~&hl_5(x_1,x_2,x_3,x_4,x_5)=-l_3(x_1,x_2,l_3(x_3,x_4,x_5))+l_3(x_3,l_3(x_1,x_2,x_4),x_5)\\
\quad &+l_3(l_3(x_1,x_2,x_3),x_4,x_5)+l_3(x_3,x_4,l_3(x_1,x_2,x_5)),
 \end{align*}
 \begin{align*}
(L6)~& l_5(h(a),x_2,x_3,x_4,x_5)=-l_3(a,x_2,l_3(x_3,x_4,x_5))+l_3(x_3,l_3(a,x_2,x_4),x_5)\\
\quad &+l_3(l_3(a,x_2,x_3),x_4,x_5)+l_3(x_3,x_4,l_3(a,x_2,x_5)),\\
(L7)~&l_5(x_1,h(a),x_3,x_4,x_5)=-l_3(x_1,a,l_3(x_3,x_4,x_5))+l_3(x_3,l_3(x_1,a,x_4),x_5)\\
\quad &+l_3(l_3(x_1,a,x_3),x_4,x_5)+l_3(x_3,x_4,l_3(x_1,a,x_5)),\\
(L8)~&l_5(x_1,x_2,h(a),x_4,x_5)=-l_3(x_1,x_2,l_3(a,x_4,x_5))+l_3(a,l_3(x_1,x_2,x_4),x_5)\\
\quad &+l_3(l_3(x_1,x_2,a),x_4,x_5)+l_3(a,x_4,l_3(x_1,x_2,x_5)),\\
(L9)~&l_5(x_1,x_2,x_3,h(a),x_5)=-l_3(x_1,x_2,l_3(x_3,a,x_5))+l_3(x_3,l_3(x_1,x_2,a),x_5)\\
\quad &+l_3(l_3(x_1,x_2,x_3),a,x_5)+l_3(x_3,a,l_3(x_1,x_2,x_5)),\\
(L10)~&l_5(x_1,x_2,x_3,x_4,h(a))=-l_3(x_1,x_2,l_3(x_3,x_4,a))+l_3(x_3,l_3(x_1,x_2,x_4),a)\\
\quad &+l_3(l_3(x_1,x_2,x_3),x_4,a)+l_3(x_3,x_4,l_3(x_1,x_2,a)),\\
(L11)~&l_3(l_5(x_1,x_2,x_3,x_4,x_5),x_6,x_7)-l_3(l_5(x_1,x_2,x_3,x_4,x_6),x_5,x_7)\\
\quad &+l_3(x_1,x_2,l_5(x_3,x_4,x_5,x_6,x_7))-l_3(x_3,x_4,l_5(x_1,x_2,x_5,x_6,x_7))\\
\quad &+l_3(x_5,x_6,l_5(x_1,x_2,x_3,x_4,x_7))-l_5(l_3(x_1,x_2,x_3),x_4,x_5,x_6,x_7)\\
\quad &-l_5(x_3,l_3(x_1,x_2,x_4),x_5,x_6,x_7)-l_5(x_3,x_4,l_3(x_1,x_2,x_5),x_6,x_7)\\
\quad &-l_5(x_3,x_4,x_5,l_3(x_1,x_2,x_6),x_7)-l_5(x_3,x_4,x_5,x_6,l_3(x_1,x_2,x_7))\\
\quad &+l_5(x_1,x_2,l_3(x_3,x_4,x_5),x_6,x_7)+l_5(x_1,x_2,x_5,l_3(x_3,x_4,x_6),x_7) \\
\quad &+l_5(x_1,x_2,x_5,x_6,l_3(x_3,x_4,x_7))-l_5(x_1,x_2,x_3,x_4,l_3(x_5,x_6,x_7))=0.
 \end{align*}
 A Lie triple 2-system is called skeletal (strict) if $h=0$
($l_5=0$).
\end{definition}
Inspired by  \cite{G23} and \cite{JS21},  we give the notion of a  Nijenhuis Lie triple 2-system.
\begin{definition}
A  Nijenhuis Lie triple 2-system consists of a Lie triple 2-system $\mathcal{G} = (T_0, T_1, h, l_3, l_5)$
and a  Nijenhuis operator $\Theta = (N_0, N_1, N_2)$ on $\mathcal{G}$, where $N_0:T_0\rightarrow T_0, N_1: T_1\rightarrow T_1$ and a  trilinear map $N_2: T_0\otimes T_0\otimes T_0\rightarrow T_1$ satisfying  the following equalities: for any $x, y, z, x_i\in T_0~(i=1,\cdot\cdot\cdot,5),a\in T_1$
\begin{small}
\begin{eqnarray*}
(a)&& N_0\circ h=h\circ N_1\\
(b)&& N_2(x, x, y) = 0,\\
(c)&& N_2(x, y, z) +N_2(y, z, x)+N_2(z, x, y)=0,\\
(d)&& N_0(l_3(N_0(x), N_0(y), z)+l_3(x, N_0(y), N_0(z))+l_3(N_0(x), y, N_0(z))\\
\quad &&-N_0  l_3(N_0(x), y, z)-N_0  l_3(x, N_0(y), z)-N_0 l_3(x, y, N_0(z)) +N^2_0 l_3(x, y, z))
-l_3(N_0(x), N_0(y), N_0(z))\\
\quad \quad \quad \quad &&=hN_2(x, y, z);
\end{eqnarray*}
\begin{eqnarray*}
(e)&&N_1(l_3(N_0(x), N_0(y), a)+l_3(x, N_0(y), N_1(a))+l_3(N_0(x), y, N_1(\a))-N_0  l_3(N_0(x), y, a)-N_0  l_3(x, N_0(y), a)\\
\quad &&-N_0 l_3(x, y, N_1(a)) +N^2_0  l_3(x, y, a))-l_3(N_1(a), N_0(x))\\
\quad \quad \quad \quad &&=N_2( x, y, h(a));\\
(f)&& l_5(N_0(x_1),N_0(x_2), N_0(x_3), N_0(x_4), N_0(x_5))+l_3(N_2(x_1, x_2, x_3), N_0(x_4), N_0(x_5))\\
\quad &&  +l_3(N_0(x_3), N_2(x_1, x_2, x_4),N_0(x_5))+l_3(N_0(x_3), N_0(x_4), N_2(x_1, x_2, x_5))\\
\quad && +N_2(l_3(N_0(x_1), N_0(x_2), x_3)+l_3(x_1, N_0(x_2), N_0(x_3))+l_3(N_0(x_1), x_2, N_0(x_3)))\\
\quad &&-N_0 l_3(N_0(x_1), x_2, x_3-N_0 l_3(x_1, N_0(x_2), x_3)-N_0 l_3(x_1, x_2, N_0(x_3)) +N^2_0 l_3(x_1, x_2, x_3), x_4, x_5)\\
\quad &&+ N_2(x_3,l_3(N_0(x_1), N_0(x_2), x_4)+l_3(x_1, N_0(x_2), N_0(x_4))+l_3(N_0(x_1), x_2, N_0(x_4))\\
\quad &&-N_0 l_3(N_0(x_1), x_2, x_4)-N_0 l_3(x_1, N_0(x_2), x_4)-N_0 l_3(x_1, x_2, N_0(x_4)) +-N_0^2 l_3(x_1, x_2, x_4),x_5 )\\
\quad &&+N_2(x_3, x_4, l_3(N_0(x_1), N_0(x_2), x_5)+l_3(x_1, N_0(x_2), N_0(x_5))+l_3(N_0(x_1), x_2, N_0(x_5))\\
\quad &&-N_0 l_3(N_0(x_1), x_2, x_5)-N_0 l_3(x_1, N_0(x_2), x_5)-N_0l_3(x_1, x_2, N_0(x_5)) +N^2_0 l_3(x_1, x_2, x_5))\\
\quad &&=l_3(N_0(x_1), N_0(x_2), N_2(x_3, x_4, x_5))+N_2(x_1, x_2, l_3(N_0(x_3), N_0(x_4), x_5)+l_3(x_3, N_0(x_4), N_0(x_5))\\
\quad &&+l_3(N_0(x_3), x_4, N_0(x_5))-N_0 l_3(N_0(x_3), x_4, x_5)-N_0 l_3(x_3, N_0(x_4), x_5)\\
\quad &&-N_0 l_3(x_3, x_4, N_0(x_5)) +N^2_0 l_3(x_3, x_4, x_5))+ N_1(l_5(x_1, x_2, x_3, x_4, x_5)).
\end{eqnarray*}
\end{small}
\end{definition}
We will denote a  Nijenhuis Lie triple 2-system by $(\mathcal{G}, \Theta)$. A  Nijenhuis Lie triple 2-system is said to
be {\bf skeletal} if $h = 0$.    A  Nijenhuis Lie triple 2-system is said to
be {\bf strict} if $l_5 = 0, N_2=0$.

Let $(\mathcal{G}, \Theta)$ be a skeletal  Nijenhuis Lie triple 2-system. Then $\mathcal{G}$ is a  Lie triple 2-system. Therefore, we have
\begin{align}
~& l_3(x,x,y)=0,~~l_3(x,x,a)=0,~~l_3(a,x,y)+l_3(x,a,y)=0\label{6.1}\\
&l_3(x,y,z)+l_3(y,z,x)+l_3(z,x,y)=0,~~l_3(x,y,a)+l_3(y,a,x)+l_3(a,x,y)=0,\\
0=&-l_3(x_1,x_2,l_3(x_3,x_4,x_5))+l_3(x_3,l_3(x_1,x_2,x_4),x_5)\nonumber\\
&+l_3(l_3(x_1,x_2,x_3),x_4,x_5)+l_3(x_3,x_4,l_3(x_1,x_2,x_5)),\\
0=&-l_3(a,x_2,l_3(x_3,x_4,x_5))+l_3(x_3,l_3(a,x_2,x_4),x_5)\nonumber\\
&+l_3(l_3(a,x_2,x_3),x_4,x_5)+l_3(x_3,x_4,l_3(a,x_2,x_5)),\\
0=&-l_3(x_1,a,l_3(x_3,x_4,x_5))+l_3(x_3,l_3(x_1,a,x_4),x_5)\nonumber\\
&+l_3(l_3(x_1,a,x_3),x_4,x_5)+l_3(x_3,x_4,l_3(x_1,a,x_5)),\\
0=&-l_3(x_1,x_2,l_3(a,x_4,x_5))+l_3(a,l_3(x_1,x_2,x_4),x_5)\nonumber\\
&+l_3(l_3(x_1,x_2,a),x_4,x_5)+l_3(a,x_4,l_3(x_1,x_2,x_5)),\\
0=&-l_3(x_1,x_2,l_3(x_3,a,x_5))+l_3(x_3,l_3(x_1,x_2,a),x_5)\nonumber\\
&+l_3(l_3(x_1,x_2,x_3),a,x_5)+l_3(x_3,a,l_3(x_1,x_2,x_5)),\\
0=&-l_3(x_1,x_2,l_3(x_3,x_4,a))+l_3(x_3,l_3(x_1,x_2,x_4),a)\nonumber\\
&+l_3(l_3(x_1,x_2,x_3),x_4,a)+l_3(x_3,x_4,l_3(x_1,x_2,a)),
 \end{align}
\begin{align}
&l_3(l_5(x_1,x_2,x_3,x_4,x_5),x_6,x_7)-l_3(l_5(x_1,x_2,x_3,x_4,x_6),x_5,x_7)
\nonumber\\&+l_3(x_1,x_2,l_5(x_3,x_4,x_5,x_6,x_7))-l_3(x_3,x_4,l_5(x_1,x_2,x_5,x_6,x_7))
\nonumber\\&+l_3(x_5,x_6,l_5(x_1,x_2,x_3,x_4,x_7))-l_5(l_3(x_1,x_2,x_3),x_4,x_5,x_6,x_7)
\nonumber\\&-l_5(x_3,l_3(x_1,x_2,x_4),x_5,x_6,x_7)-l_5(x_3,x_4,l_3(x_1,x_2,x_5),x_6,x_7)
\nonumber\\&-l_5(x_3,x_4,x_5,l_3(x_1,x_2,x_6),x_7)-l_5(x_3,x_4,x_5,x_6,l_3(x_1,x_2,x_7))
\nonumber\\&+l_5(x_1,x_2,l_3(x_3,x_4,x_5),x_6,x_7)+l_5(x_1,x_2,x_5,l_3(x_3,x_4,x_6),x_7) \nonumber\\
&+l_5(x_1,x_2,x_5,x_6,l_3(x_3,x_4,x_7))-l_5(x_1,x_2,x_3,x_4,l_3(x_5,x_6,x_7))=0.\label{6.13}
 \end{align}
\begin{theorem}
 There is a one-to-one correspondence between skeletal Nijenhuis Lie triple 2-systems and 5-cocycles
of Nijenhuis Lie triple systems.
\end{theorem}
\begin{proof}
Let $(\mathcal{G}, \Theta)$ be a skeletal   Lie triple 2-system. Then $(T_0, l_3, N_0)$ is a  Nijenhuis Lie triple system. Define linear maps
$f: (T_0\wedge T_0) \wedge (T_0\wedge T_0\wedge T_0 )\rightarrow T_1$ and $\theta: T_0 \wedge T_0\wedge T_0\rightarrow T_1$ by
\begin{eqnarray*}
& f(x_1, x_2, x_3, x_4, x_5)=l_5(x_1, x_2, x_3, x_4, x_5),\\
& \theta(x, y, z)=N_2(x, y, z).
\end{eqnarray*}
By (e) in Definition 6.2, we have, we have
\begin{eqnarray}
&& f(N_0(x_1),N_0(x_2), N_0(x_3), N_0(x_4), N_0(x_5))+l_3(N_2(x_1, x_2, x_3), N_0(x_4), N_0(x_5))\nonumber\\
\quad &&  +l_3(N_0(x_3), \theta(x_1, x_2, x_4),N_0(x_5))+l_3(N_0(x_3), N_0(x_4), \theta(x_1, x_2, x_5))\nonumber\\
\quad && +\theta(l_3(N_0(x_1), N_0(x_2), x_3)+l_3(x_1, N_0(x_2), N_0(x_3))+l_3(N_0(x_1), x_2, N_0(x_3))-N_0 l_3(N_0(x_1), x_2, x_3)\nonumber\\
\quad &&-N_0 l_3(x_1, N_0(x_2), x_3)-N_0 l_3(x_1, x_2, N_0(x_3)) +N_0^2 l_3(x_1, x_2, x_3), x_4, x_5)\nonumber\\
\quad &&+ \theta(x_3,l_3(N_0(x_1), N_0(x_2), x_4)+l_3(x_1, N_0(x_2), N_0(x_4))+l_3(N_0(x_1), x_2, N_0(x_4)) \label{6.14}\\
\quad &&-N_0 l_3(N_0(x_1), x_2, x_4)-N_0 l_3(x_1, N_0(x_2), x_4)-N_0 l_3(x_1, x_2, N_0(x_4)) +N_0^2 l_3(x_1, x_2, x_4),x_5 )\nonumber\\
\quad &&+N_2(x_3, x_4, l_3(N_0(x_1), N_0(x_2), x_5)+l_3(x_1, N_0(x_2), N_0(x_5))+l_3(N_0(x_1), x_2, N_0(x_5))\nonumber\\
\quad &&-N_0 l_3(N_0(x_1), x_2, x_5)-N_0 l_3(x_1, N_0(x_2), x_5)-N_0 l_3(x_1, x_2, N_0(x_5)) +N_0^2 l_3(x_1, x_2, x_5))\nonumber\\
\quad &&-l_3(N_0(x_1), N_0(x_2), N_2(x_3, x_4, x_5))-\theta(x_1, x_2, l_3(N_0(x_3), N_0(x_4), x_5)+l_3(x_3, N_0(x_4), N_0(x_5))\nonumber\\
\quad &&+l_3(N_0(x_3), x_4, N_0(x_5))-N_0 l_3(N_0(x_3), x_4, x_5)-N_0 l_3(x_3, N_0(x_4), x_5)\nonumber\\
\quad &&-N_0 l_3(x_3, x_4, N_0(x_5)) +N_0^2 l_3(x_3, x_4, x_5))- N_1(f(x_1, x_2, x_3, x_4, x_5))=0.\nonumber
\end{eqnarray}
 By (\ref{6.13}) and (\ref{6.14}),  we have $d^5(f, \theta)=0$. Thus,  $(f, \theta)$ is a 5-cocycle of  Nijenhuis Lie triple systems.

The proof of the other direction is similar, and  we are done.
\end{proof}

Now we turn to the study on strict Nijenhuis Lie triple 2-systems. First we introduce the notion of crossed modules of
Nijenhuis Lie triple systems.
\begin{definition}
A crossed module of  Nijenhuis Lie triple systems is a sextuple $((T{_0}, [\cdot, \cdot, \cdot]_{T_0}),(T_1, [\cdot, \cdot,\\ \cdot]_{T_1}), h,\Lambda,  N_0, N_1)$
where $(T_0, [\cdot, \cdot, \cdot]_{T_0}, R_0)$ is a   Nijenhuis Lie triple system and $(T_1, [\cdot, \cdot, \cdot]_{T_1})$ is a  Lie triple system, $h: T_1\rightarrow T_0$ is a homomorphism of Lie triple systems,   and $(T_1, \Lambda,  N_1)$ is a Nijenhuis representation  over $(T_0, [\cdot, \cdot, \cdot]_{T_0}, N_0)$ such that for all $x, y \in T_0$ and $a, b, c \in T_1$, the following
equalities are satisfied:
\begin{align*}
N_0\circ h&=h\circ N_1,\\
h(\Lambda(x,y)a)&=[h(a),x,y]_{T_0},\\
\Lambda(h(a),h(b))c&=[c,a,b]_{T_1}.
\end{align*}
\end{definition}
Let $(\mathcal{G}, \Theta)$ be a strict Nijenhuis Lie triple 2-system. Then $\mathcal{G}$ is a  strict Lie triple 2-system. Therefore, we have
\begin{align}
& l_3(x,x,y)=0,~~l_3(x,x,a)=0,~~l_3(a,x,y)+l_3(x,a,y)=0, \label{6.15}\\
& hl_3(a,y,z)=l_3(h(a),y,z),\label{6.16}\\
&l_3(h(a),b,x)=l_3(a,h(b),x),~~l_3(h(a),x,b)=l_3(a,x,h(b)),~~l_3(x,h(a),b)=l_3(x,a,h(b)),\label{6.17}\\
&l_3(x,y,z)+l_3(y,z,x)+l_3(z,x,y)=0,~~l_3(x,y,a)+l_3(y,a,x)+l_3(a,x,y)=0,\label{6.18}\\
0=&-l_3(x_1,x_2,l_3(x_3,x_4,x_5))+l_3(x_3,l_3(x_1,x_2,x_4),x_5)\nonumber\\
&+l_3(l_3(x_1,x_2,x_3),x_4,x_5)+l_3(x_3,x_4,l_3(x_1,x_2,x_5)),\label{6.19}\\
0=&-l_3(a,x_2,l_3(x_3,x_4,x_5))+l_3(x_3,l_3(a,x_2,x_4),x_5)\nonumber\\
&+l_3(l_3(a,x_2,x_3),x_4,x_5)+l_3(x_3,x_4,l_3(a,x_2,x_5)),\label{6.20}\\
0=&-l_3(x_1,a,l_3(x_3,x_4,x_5))+l_3(x_3,l_3(x_1,a,x_4),x_5)\nonumber\\
&+l_3(l_3(x_1,a,x_3),x_4,x_5)+l_3(x_3,x_4,l_3(x_1,a,x_5)),\label{6.21}\\
0=&-l_3(x_1,x_2,l_3(a,x_4,x_5))+l_3(a,l_3(x_1,x_2,x_4),x_5)\nonumber\\
&+l_3(l_3(x_1,x_2,a),x_4,x_5)+l_3(a,x_4,l_3(x_1,x_2,x_5)),\label{6.22}\\
0=&-l_3(x_1,x_2,l_3(x_3,a,x_5))+l_3(x_3,l_3(x_1,x_2,a),x_5)\nonumber\\
&+l_3(l_3(x_1,x_2,x_3),a,x_5)+l_3(x_3,a,l_3(x_1,x_2,x_5)),\label{6.23}\\
0=&-l_3(x_1,x_2,l_3(x_3,x_4,a))+l_3(x_3,l_3(x_1,x_2,x_4),a)\nonumber\\
&+l_3(l_3(x_1,x_2,x_3),x_4,a)+l_3(x_3,x_4,l_3(x_1,x_2,a)).\label{6.24}
 \end{align}
for all $x, y, z, x_i(i=1,...5)\in T_0$ and $a, b\in T_1$.
\begin{theorem}
 There is a one-to-one correspondence between strict Nijenhuis Lie triple 2-systems and crossed modules of
Nijenhuis Lie triple systems.
\end{theorem}
\begin{proof}
Let $(\mathcal{G}, \Theta)$ be a strict  Nijenhuis Lie triple 2-system. We construct a crossed module of
strict  Nijenhuis Lie triple system as follows. Obviously, $(T_0, [\cdot, \cdot, \cdot]_{T_0}, N_0)$ is a  Nijenhuis Lie triple system. Define brackets  $[\cdot,\cdot,\cdot]_{T_0}, [\cdot, \cdot, \cdot]_{T_1}$ on $T_0, T_1$ respectively as:
\begin{align}
[x, y, z]_{T_0}&=l_3(x, y, z), \nonumber\\
[a, b, c]_{T_1}&=l_3(h(a),h(b),c)=l_3(h(a),b,h(c))=l_3(a,h(b),h(c)).\label{6.25}
\end{align}
By conditions (\ref{6.15}), (\ref{6.18}), (\ref{6.19}) and  (\ref{6.25}). Then    $(T_1, [\cdot, \cdot, \cdot]_{T_1})$ is a Lie triple system.  Obviously, we deduce that $h$ is a
homomorphism between Lie triple systems. Define $\Lambda: T_0\wedge T_0  \rightarrow \mathfrak{gl}(T_1)$ by
\begin{align}
\Lambda(h(a),h(b))c&=l_3(c,h(a),h(b))=[c,a, b ]_{T_1}. \label{6.26}
\end{align}
By conditions (d),  (e) and (\ref{6.20}-\ref{6.24}), (\ref{6.26}),  the triple $(T_1, \Lambda, N_1)$ is a  representation  over the Nijenhuis Lie triple system $(T_0, [\cdot, \cdot, \cdot]_{T_0}, N_0)$.
Thus $((T_0, [\cdot, \cdot, \cdot]_{T_0}), (T_1, [\cdot, \cdot, \cdot]_{T_1}), h, \theta, N_0, N_1)$ is a crossed module of  Nijenhuis Lie triple system.

Conversely, a crossed module of Nijenhuis Lie triple system $((T_0, [\cdot, \cdot, \cdot]_{T_0}), (T_1, [\cdot, \cdot, \cdot]_{T_1}), h, \Lambda,N_0, N_1)$  gives rise to a strict  Nijenhuis Lie triple 2-system $(T_0,T_1, h, l_3, l_5 = 0, N_0, N_1, N_2=0)$, where $l_3 : T_i \wedge T_j \wedge T_k \rightarrow T_{i+j+k} , 0 \leq i + j+k \leq 1$ is given by
\begin{eqnarray*}
l_3(x, y, z) = [x, y, z]_{T_0}, l_3(c,h(a),h(b))=[c,a, b ]_{T_1}=\Lambda(h(a),h(b))(c).
\end{eqnarray*}
Direct verification shows that $(T_0, T_1, h, l_3, l_5=0, N_0, N_1, N_2=0)$  is a strict  Nijenhuis Lie triple 2-system.
\end{proof}

 \begin{center}
 {\bf ACKNOWLEDGEMENT}
 \end{center}
 Shuangjian Guo is  supported  by Natural Science Foundation of China ( No. 12161013) and   Guizhou Provincial Basic Research Program (Natural Science) (No. ZK[2023]025). Ripan Saha (Corresponding Author) is supported by the Science and Engineering Research Board (SERB), Department of Science and Technology (DST), Govt. of India(Grant Number-CRG/2022/005332).

\end{document}